\documentclass{proc-l}

\newtheorem{theorem}{Theorem}[section]
\newtheorem{lemma}[theorem]{Lemma}
\newtheorem{corollary}[theorem]{Corollary}

\theoremstyle{remark}
\newtheorem{remark}[theorem]{Remark}

\numberwithin{equation}{section}

\newcommand{\CC}{\mathbb{C}}
\newcommand{\RR}{\mathbb{R}}
\newcommand{\NN}{\mathbb{N}}
\newcommand{\ZZ}{\mathbb{Z}}
\newcommand{\Zpl}{\ZZ_+}

\newcommand{\set}[1]{\underline{#1}}
\newcommand{\setn}[1]{\underline{#1}_0}
\newcommand{\card}[1]{|#1|}

\newcommand{\dd}{{\mathrm{d}}}
\newcommand{\ee}{{\mathrm{e}}}
\newcommand{\ii}{{\mathrm{i}}}
\newcommand{\pii}{{\mathrm{\pi}}}
\newcommand{\ttt}{{\mathrm{t}}}
 
\newcommand{\dirac}{\delta}                 

\newcommand{\ab}[1]{\vert#1\vert}           
\newcommand{\Ab}[1]{\Big\vert#1\Big\vert}   

\newcommand{\binomial}[2]{\genfrac{(}{)}{0pt}{}{#1}{#2}}

\newcommand{\cala}{\mathcal{A}}
\newcommand{\calf}{\mathcal{F}}
\newcommand{\cals}{\mathcal{S}}
\newcommand{\calt}{\mathcal{T}}
\newcommand{\coeff}[2]{\mathrm{Coeff}(#1;\,#2)}
\newcommand{\Coeff}[2]{\mathrm{Coeff}\Big(#1;\,#2\Big)}

\newcommand{\dtv}[2]{d_{\mathrm{TV}}(#1,#2)}
\newcommand{\dpv}{d_{\mathrm{PV}}}
 
\newcommand{\Per}[1]{\mathrm{Per}(#1)}

\newcommand{\Repart}{\mathrm{Re}}           

\newcommand{\leer}[1]{}

\begin{document}

\title[Approximate de~Finetti representation and permanents]{On 
Bobkov's approximate de~Finetti representation via 
approximation of permanents of complex rectangular matrices}

\author{Bero Roos}
\address{FB IV -- Department of Mathematics, 
University of Trier, 
54286 Trier, Germany.}
 
\email{bero.roos@uni-trier.de}


\date{December 14, 2013: Revised version}

\commby{Walter Craig}

\begin{abstract}
Bobkov (J.\ Theoret.\ Probab.\ \textbf{18}(2) (2005) 399--412)
investigated an approximate de~Finetti representation 
for probability measures, on product measurable spaces, which are 
symmetric under permutations of coordinates. One of the main results 
of that paper was an explicit approximation bound for permanents of 
complex rectangular matrices, which was shown by a somewhat complicated 
induction argument. In this paper, we indicate how to avoid the 
induction argument using an (asymptotic) expansion. Our approach 
makes it possible to give new explicit higher order approximation 
bounds for such permanents and in turn for the probability measures 
mentioned above. 
\end{abstract}

\maketitle

\section{Introduction} 
Suppose that $X:=(X_1,X_2,X_3,\dots)$ is an infinite 
exchangeable sequence of random variables on a probability space 
$(\Omega,\cala,P)$ with values in a 
measurable space $(S,\cals)$, that is, the distribution 
$P^{X}$ of 
$X$ on the infinite product measurable space 
$(S^\infty,\cals^{\otimes\infty})$ is invariant under
permutations of a finite number of coordinates. 
The de~Finetti Theorem says that, 
under mild assumptions on the space $(S,\cals)$, there is
a probability space $(T,\calt,\nu)$ and a 
Markov kernel $\mu:\,T\times\cals\longrightarrow[0,1]$, 
$(t,A)\mapsto\mu_t(A)$ such that 
\begin{align*}
P^{X}=\int_T(\mu_t)^{\otimes\infty}\,\dd\nu(t).
\end{align*}
For instance, it suffices to assume that $(S,\cals)$ is a Borel (or 
standard) measurable space, i.e.\ Borel isomorphic to some 
Borel measurable subset of $\RR$ (see Hewitt and Savage \cite{HS55} or 
Diaconis and Freedman \cite{DF80}).  

For a finite exchangeable sequence, an analogous representation does 
not generally hold, but there are approximate de~Finetti results. 
In what follows, let $N\in\NN$, $n\in\set{N}=\{1,\dots,N\}$, and let
$Y_N=(X_1,\dots,X_N)$ be an exchangeable 
family of $S$-valued random variables, that is, the distribution 
$P^{Y_N}$ of $Y_N$, defined on $(S^N,\cals^{\otimes N})$,
is invariant under permutation of coordinates.  
Let $Q_1$ be the probability measure on 
$(S^n,\cals^{\otimes n})$ defined by 
\begin{align*}
Q_1(A)=\int\Big(\frac{1}{N}\sum_{j=1}^N
\dirac_{X_j(\omega)}\Big)^{\otimes n}(A)\,\dd P(\omega)
\end{align*}
for $A\in\cals^{\otimes n}$, where $\dirac_x$ denotes the Dirac 
measure at the point $x\in S$. 
In other words, $Q_1$ is the $P$-expectation of the $n$-th power of an 
empirical measure on $(S^N,\cals^{\otimes N})$. 
The following results can be found in Diaconis and Freedman~\cite{DF80}.
They showed that 
\begin{align}\label{eq719}
\dtv{P^{Y_n}}{Q_1}\leq 1-\frac{N!}{(N-n)!N^n}
\leq \frac{n(n-1)}{2N},
\end{align}
where $\dtv{R}{R'}=\sup_{A\in\cals^{\otimes n}}\ab{R(A)-R'(A)}$ 
denotes the total variation distance between finite signed measures 
$R$ and $R'$ on $(S^n,\cals^{\otimes n})$. Hence, if $\frac{n^2}{N}$ 
is small then $P^{Y_n}$ has an approximate de~Finetti 
representation $Q_1$. It turned out that, in general, the bound (\ref
{eq719}) is sharp. However, if~$S$ is finite and of cardinality 
$\card{S}=d\in\NN$, then the nice inequality
\begin{align}\label{eq1755}
\dtv{P^{Y_n}}{Q_1}\leq \frac{dn}{N}
\end{align}
is available, which, in the case of finite $S$, is 
better than (\ref{eq719}) if $d$ is sufficiently small compared 
to~$n$.

On the other hand, it is possible to obtain similar good 
bounds in the general case if the total variation distance is 
replaced by a weaker metric. Let $\calf_n$ be the set 
of all functions $f:\,S^n\longrightarrow \CC$ such that measurable 
$f_1,\dots,f_n:\,S\longrightarrow \CC$ exist with 
$\ab{f_k(x_k)}\leq1$ for $k\in\set{n}$ and 
$f(x)=\prod_{k=1}^nf_k(x_k)$ for all $x=(x_1,\dots,x_n)\in S^n$. 
We write $f=\bigotimes_{k=1}^n f_k$. Furthermore, let
$\set{N}_{\neq}^n=\{(j_1,\dots,j_n)\in\set{N}^n\,|\,
j_k\neq j_\ell \mbox{ for all }k,\ell\in\set{n}\mbox{ with }
k\neq \ell\}$.

Bobkov \cite{bob05} showed in his Theorem 1.1 (see also p.~405 there) 
the inequality 
\begin{align}\label{eq41887}
\sup_{f\in\calf_n}\Ab{\int f\,\dd (P^{Y_n}-Q_1)} \leq C\frac{n}{N}
\quad \mbox{ with } C=16.
\end{align}
For the proof, he used the representation
\begin{align}
\lefteqn{
\int f\,\dd (P^{Y_n}-Q_1)
=\int\Big(\prod_{k=1}^nf_k(X_k(\omega)) - 
\int f\,\dd\Big(\frac{1}{N}\sum_{j=1}^N
\dirac_{X_j(\omega)}\Big)^{\otimes n}\Big)\,\dd P(\omega)}
\qquad\qquad\nonumber\\
&=\int\Big(\frac{(N-n)!}{N!}\sum_{j\in\set{N}_{\neq}^n}
\prod_{k=1}^nf_k(X_{j_k}(\omega)) 
-\prod_{k=1}^n\Big(\frac{1}{N}\sum_{j=1}^N
f_k(X_j(\omega))\Big)\Big)\,\dd P(\omega)\label{eq892762}
\end{align}
for $f=\bigotimes_{k=1}^n f_k\in\calf_n$ 
and a remarkable approximation result for 
permanents of complex rectangular matrices
(see Theorem A below), which he proved by using 
a somewhat complicated induction argument. The permanent of a complex 
rectangular matrix $Z=(z_{j,k})\in\CC^{N\times n}$ with $N\in\NN$ 
and $n\in\set{N}$ is defined by
\begin{align*}
\Per{Z}:=\sum_{j\in\set{N}_{\neq}^n}\prod_{k=1}^nz_{j_k,k}.
\end{align*}
For general properties of permanents, we refer the reader to 
Minc \cite{min78} and Cheon and Wanless \cite{CW05}.\medskip 

\noindent
\textbf{Theorem A.} (Bobkov \cite[Theorem 2.1]{bob05})
\textit{
Let $N\in\NN$, $n\in\set{N}$ and $Z=(z_{j,k})\in\CC^{N\times n}$. 
For $j\in\set{N}$ and $k\in\set{n}$, we assume that
$\ab{z_{j,k}}\leq1$  and set
$\widetilde{z}_k=\frac{1}{N}\sum_{j=1}^Nz_{j,k}$. Then
\begin{align}\label{eq2645}
\Ab{\frac{(N-n)!}{N!}\Per{Z}-\prod_{k=1}^n\widetilde{z}_k}\leq 
C\frac{n}{N} \quad \mbox{ with }C=16.
\end{align}
}  

From Proposition 4.1 in Bobkov \cite{bob04} it 
follows that (\ref{eq41887}) and (\ref{eq2645}) hold with the better 
constant $C=6$ if $z_{j,k}=z_{j,1}$ for all $j\in\set{N}$ and 
$k\in\set{n}$. However, Theorem~\ref{th12} below shows that $C$ can 
always be taken smaller than 3.57. 

For two finite signed measures $R$ and $R'$ on 
$(S^n,\cals^{\otimes n})$, let 
\begin{align*}
\dpv(R,R')
=\sup_{A_1,\dots,A_n\in\cals}\ab{R(A_1\times\ldots\times A_n)
-R'(A_1\times\ldots\times A_n)}
\end{align*}
denote the so-called product variation between $R$ and $R'$. 
Obviously $\dpv$ is a metric on the set of all finite signed
measures on $(S^n,\cals^{\otimes n})$. Furthermore,  
\begin{align*}
\dpv(R,R')\leq\sup_{f\in\calf_n}\Ab{\int f\,\dd (R-R')}. 
\end{align*}
Therefore (\ref{eq41887}) and the inequalities of Theorem \ref{th363}
below imply bounds for $\dpv$.

In the next section, we present refinements of (\ref{eq2645}), see 
Theorems \ref{th2994}, \ref{th12} and Corollary \ref{cor387}. The 
latter together with (\ref{eq892762}) and a similar representation
implies Theorem \ref{th363} below, the first part of which is better 
than (\ref{eq41887}) with $C=3.57$ if $\frac{n}{N}\leq\frac{1}{2}$. 
The second part shows that, if $n\geq 2$ and in turn $N\geq 2$,
a more accurate approximation of $P^{Y_n}$ by a finite signed measure $Q_2$ on 
$(S^n,\cals^{\otimes n})$ is possible, where
\begin{align*}
Q_2(A) 
&=Q_1(A)-\frac{1}{N(N-1)}\sum_{K\subseteq\set{n}:\,\card{K}=2}
\sum_{j=1}^N\int\Big(\bigotimes_{k\in \set{n}}
R_{j,k,K}(\omega)\Big)(A)\,\dd P(\omega) 
\end{align*} 
for $A\in\cals^{\otimes n}$ and
\begin{align*}
R_{j,k,K}(\omega)=\left\{
\begin{array}{ll}
\dirac_{X_j(\omega)}-\frac{1}{N}
\sum_{\ell=1}^N\dirac_{X_\ell(\omega)},& 
\mbox{ if } k\in K,\\
\frac{1}{N}\sum_{\ell=1}^N\dirac_{X_\ell(\omega)},&
\mbox{ if } k\in\set{n}\setminus K.
\end{array}
\right.
\end{align*} 
\begin{theorem}\label{th363}
Under the assumptions above and if $\frac{n}{N}<1$, we have 
\begin{align} 
\sup_{f\in\calf_n}
\Ab{\int f\,\dd (P^{Y_n}-Q_1)}
&\leq  \frac{n}{N}
+2.12\,\frac{(\frac{n}{N})^{3/2}}{(1-\frac{n}{N})^{3/4}},
\label{eq71665}\\
\sup_{f\in\calf_n}
\Ab{\int f\,\dd (P^{Y_n}-Q_2)}
&\leq\sqrt{3}\Big(\frac{n}{N}\Big)^{3/2}
+2.27\,\frac{(\frac{n}{N})^2}{(1-\frac{n}{N})^{3/4}}, \quad
\mbox{ if }n\geq 2.\label{eq71666}
\end{align}
\end{theorem}
Higher order results are also possible using Theorem~\ref{th2994}  
or Theorem~\ref{th12} below. We omit the details. 
\section{Approximation of permanents}
For $n\in\NN$, the  indeterminate $x=(x_1,\dots,x_n)$ and 
$r\in\Zpl^n=\{0,1,2,\dots\}^n$, we set 
$x^r=\prod_{k\in\set{n}}x_k^{r_k}$ and 
write $a_r=\coeff{x^r}{\sum_{s\in\Zpl^n}a_sx^s}$ for the 
coefficient of $x^r$ in the formal power series 
$\sum_{s\in\Zpl^n}a_sx^s$, 
$(a_s\in\CC)$. Sometimes $y$ will be our indeterminate. 
However, the symbols $x$ and $y$ may have other meanings as 
indicated below. In what follows, we use the simple fact that, for 
$N\in\NN$, $n\in\set{N}$ and $Z=(z_{j,k})\in\CC^{N\times n}$, 
\begin{align}\label{eq1326}
\Per{Z}= 
\Coeff{x_1\cdots x_n}{\prod_{j=1}^N
\Big(1+\sum_{k=1}^nz_{j,k}x_k\Big)}.
\end{align}
Furthermore, if additionally $Z$ has identical columns,
i.e.\ $z_{j,k}=z_{j,1}$ for all $j\in\set{N}$ and $k\in\set{n}$,
then 
\begin{align}\label{eq1326b}
\Per{Z}=n!\,\Coeff{y^n}{\prod_{j=1}^N(1+z_{j,1}y)}.
\end{align}

The main result of this section is Theorem~\ref{th2994} below and 
requires the following lemmas, the first of which plays a prominent 
role in the theory of polynomials over infinite dimensional spaces. 
Its proof is due to H\"ormander \cite[Theorem 4]{hoe54};
see also Harris \cite{har96} and
Dineen \cite[Proposition 1.44 and the notes on page 79]{din99}. 
However, first versions for real spaces were already shown in 
Kellogg \cite{kel28} and Banach \cite{ban38}.
\begin{lemma}\label{la61966}
Let $n\in\NN$, 
$E$ be a complex Hilbert space, $F$ be a complex Banach space,
$g:\,E^n\longrightarrow F$ be $n$-linear (i.e.\ linear in each 
component), continuous and symmetric in its arguments.
Let $\widehat{g}(x)=g(x,\dots,x)$ for $x\in E$ and
\begin{align*}
\|g\|&=\sup\{\|g(x_1,\dots,x_n)\|\,|\,x_i\in E,\|x_i\|\leq1
\mbox{ for each }i\in\set{n}\}\\
\|\widehat{g}\|&=\sup\{\|\widehat{g}(x)\|\,|\,x\in E,\,\|x\|\leq 1\}.
\end{align*}
Then 
$\|g\|=\|\widehat{g}\|$.
\end{lemma}

The proof of the next lemma uses Lemma \ref{la61966} and
the Cauchy integral formula. We note that
the more complicated Lemma 3 in \cite{roo01} only yields 
a weaker result under the assumptions used here. We always set $0^0=1$.
\begin{lemma}\label{la2655}
Let $N\in\NN$, $n\in\set{N}$ and $A=(a_{j,k})\in\CC^{N\times n}$. For 
each $k\in\set{n}$, we assume that $\sum_{j=1}^Na_{j,k}=0$ and
set $\alpha_k=\frac{1}{N}\sum_{j=1}^N\ab{a_{j,k}}^2$. Then
we have 
\begin{align}\label{eq45}
\ab{\Per{A}}
&\leq\frac{n!\,N^{N/2}}{(N-n)^{(N-n)/2}n^{n/2}}
\prod_{k=1}^n\sqrt{\alpha_k}.
\end{align}
\end{lemma}
\begin{proof}
We may assume that $\alpha_k\neq0$ for each $k\in\set{n}$. 
Let $E=\{x={^{\ttt}}(x_1,\dots,x_N)\in\CC^{N\times 1}\,|\,
\sum_{j=1}^Nx_j=0\}$ be equipped with the standard inner product 
and consider $F=\CC$, 
$g:\,E^n\longrightarrow F$, $g(x^{(1)},\dots,x^{(n)})
=\Per{x^{(1)},\dots,x^{(n)}}$
for $x^{(1)},\dots,x^{(n)}\in E$, where ``$\ttt$'' denotes 
transposition. 
It is easily seen that Lemma \ref{la61966} can be applied, which gives
$\ab{\Per{A}}\leq \|\widehat{g}\|N^{n/2}\prod_{k=1}^n\sqrt{\alpha_k}$. 
Using (\ref{eq1326b}), we obtain for $x\in E$ with $\|x\|\leq 1$ 
and arbitrary $r\in(0,\infty)$ that
\begin{align*} 
\ab{\widehat{g}(x)}
&=\frac{n!}{2\pii r^n}\Ab{\int_{-\pii}^\pii\ee^{-\ii nt}
\prod_{j=1}^N\Big(1+x_jr\ee^{\ii t}\Big)\dd t}\\
&\leq\frac{n!}{r^n}\sup_{t\in[-\pii,\pii]}
\prod_{j=1}^N\ab{1+x_{j}r\ee^{\ii t}} 
\leq\frac{n!}{r^n} \Big(1+\frac{r^2}{N}\Big)^{N/2};
\end{align*}
the last inequality follows from the inequality between arithmetic 
and geometric means. Indeed, for $w\in E$, we have 
\begin{align*}
\prod_{j=1}^N\ab{1+w_j}
&=\Big(\prod_{j=1}^N(1+2\Repart(w_j)+\ab{w_j}^2)\Big)^{1/2}
\leq\Big(1+\frac{1}{N}\sum_{j=1}^N\ab{w_j}^2\Big)^{N/2},
\end{align*}
where $\Repart(w_j)$ denotes the real part of $w_j$. 
Let $\varepsilon\in(0,\infty)$ and
\begin{equation*}
r=\Big(\frac{nN}{N-n+\varepsilon}\Big)^{1/2}.
\end{equation*}
Letting $\varepsilon\to0$ yields
$\|\widehat{g}\|\leq\frac{n!N^{(N-n)/2}}{(N-n)^{(N-n)/2}n^{n/2}}$
and the result is shown. 
\end{proof}

\begin{remark} 
Inequality (\ref{eq45}) can be viewed as a Hadamard type inequality 
for permanents of matrices with zero column sums. 
Another inequality of this type is 
\begin{equation}\label{eq617}
\ab{\Per{Z}}
\leq N!\prod_{k=1}^N\Big(\frac{1}{N}
\sum_{j=1}^N\ab{z_{j,k}}^2\Big)^{1/2}, 
\end{equation}
which is valid for general quadratic matrices 
$Z=(z_{j,k})\in\CC^{N\times N}$ with $N\in\NN$.
Carlen \emph{et al.} \cite{CLL06} gave two proofs of (\ref{eq617}), which, however, also follows directly from Lemma \ref{la61966} 
together with the inequality between arithmetic and geometric means.

Inequality (\ref{eq617}) can be used to derive an alternative bound for 
the left-hand side of (\ref{eq45}) as follows. Consider the 
assumptions of Lemma \ref{la2655} and define 
$Z=(z_{j,k})\in\CC^{N\times N}$ with
$z_{j,k}=a_{j,k}$ for $j\in\set{N}$, $k\in\set{n}$ and  
$z_{j,k}=1$ for $j\in\set{N}$, $k\in\set{N}\setminus\set{n}$. Then
\begin{equation}\label{eq618}
\ab{\Per{A}}=\frac{\ab{\Per{Z}}}{(N-n)!}
\leq\frac{N!}{(N-n)!}\prod_{k=1}^n\sqrt{\alpha_k}.
\end{equation}
However, it turns out that 
(\ref{eq45}) is always better than the inequality in (\ref{eq618}),
since
\begin{equation*}
\frac{N^{N}}{(N-n)^{N-n}n^{n}}
\leq \Big(\prod_{m=1}^{N-n}\frac{N-m+1}{N-n-m+1}\Big) 
\Big(\prod_{m=1}^n\frac{N-m+1}{n-m+1}\Big)
=\binomial{N}{n}^2.
\end{equation*}
\end{remark}

\begin{lemma}\label{la124376}
Let $n\in\NN$, $m\in\setn{n}=\{0,\dots,n\}$ 
and $w_{1,k},w_{2,k}\in\CC$ for $k\in\set{n}$. Then 
\begin{equation*} 
\Ab{\Coeff{y^m}{\prod_{k=1}^n(w_{1,k}+w_{2,k}y)}}
\leq \binomial{n}{m}
\Big(\frac{1}{n}\sum_{k=1}^n \ab{w_{2,k}}^2\Big)^{m/2}
\Big(\frac{1}{n}\sum_{k=1}^n\ab{w_{1,k}}^2\Big)^{(n-m)/2}.
\end{equation*}
\end{lemma}
\begin{proof}
Using Cauchy's inequality, we obtain
\begin{align*}
\lefteqn{\Ab{\Coeff{y^m}{\prod_{k=1}^n(w_{1,k}+w_{2,k}y)}}
=\Ab{
\sum_{K\subseteq\set{n}:\,\card{K}=m}
\Big(\prod_{k\in K}w_{2,k}\Big)\prod_{k\in\set{n}\setminus K}
w_{1,k}}}\\
&\leq\Big(\sum_{K\subseteq\set{n}:\,\card{K}=m}
\prod_{k\in K}\ab{w_{2,k}}^2\Big)^{1/2}
\Big(\sum_{K\subseteq\set{n}:\,\card{K}=m}
\prod_{k\in\set{n}\setminus K}
\ab{w_{1,k}}^2\Big)^{1/2}\\
&= \Coeff{y^m}{\prod_{k=1}^n(1+
\ab{w_{2,k}}^2 y)}^{1/2}
\Coeff{y^{n-m}}{\prod_{k=1}^n(1+
\ab{w_{1,k}}^2 y)}^{1/2}.
\end{align*}
The assertion now follows from a result due to Maclaurin, 
which says that if $g_1,\dots,g_n\in[0,\infty)$, then
$\big(\frac{1}{\binomial{n}{\ell}}\coeff{y^\ell}{
\prod_{k=1}^n(1+g_ky)}\big)^{1/\ell}$
is non-increasing in $\ell\in\set{n}$, see 
Hardy \emph{et al.} \cite[Theorem~52, page~52]{HLP52}. 
\end{proof}
\begin{lemma}\label{la31765}
Let $n,N\in\NN$, $m\in\Zpl$ with $m\leq \min\{n,N\}$, 
$(a_{j,k})\in\CC^{N\times n}$ with $\sum_{j=1}^Na_{j,k}=0$ 
for all $k\in\set{n}$, $b\in\CC^n$, 
$\alpha=\frac{1}{nN}\sum_{j=1}^N\sum_{k=1}^n \ab{a_{j,k}}^2$,
 $\beta=\frac{1}{n}\sum_{k=1}^n\ab{b_k}^2$.
Then
\begin{align*}
\Ab{\Coeff{x_1\cdots x_n}{\Big(\sum_{k=1}^nb_kx_k\Big)^{n-m}
\prod_{j=1}^N\Big(1+\sum_{k=1}^n a_{j,k}x_k\Big)}} 
\leq  \frac{n!N^{N/2}\alpha^{m/2}\beta^{(n-m)/2}}{
(N-m)^{(N-m)/2}m^{m/2}}.
\end{align*}

\end{lemma}
\begin{proof}
Let 
$\alpha_k=\frac{1}{N}\sum_{j=1}^N\ab{a_{j,k}}^2$, $(k\in\set{n})$.
An application of Lemma~\ref{la2655} gives
\begin{align*}
\lefteqn{\frac{1}{(n-m)!}
\Ab{\Coeff{x_1\cdots x_n}{
\Big(\sum_{k=1}^nb_kx_k\Big)^{n-m}
\prod_{j=1}^N\Big(1+\sum_{k=1}^n a_{j,k}x_k\Big)}}} \\
&=\Ab{\sum_{K\subseteq\set{n}:\;\card{K}=n-m}
\Coeff{x_1\cdots x_n}{\Big(\prod_{k\in K}(b_kx_k)\Big)
\prod_{j=1}^N\Big(1+\sum_{k=1}^n a_{j,k}x_k\Big)}} \\
&=\Ab{\sum_{K\subseteq\set{n}:\;\card{K}=n-m}
\Coeff{\prod_{k\in\set{n}\setminus K}x_k}{
\prod_{j=1}^N\Big(1+\sum_{k\in\set{n}\setminus K}
a_{j,k}x_k\Big)}\prod_{k\in K}b_k} \\
&\leq\frac{m!N^{N/2}}{(N-m)^{(N-m)/2}m^{m/2}}
\sum_{K\subseteq\set{n}:\;\card{K}=n-m}
\Big(\prod_{k\in\set{n}\setminus K} \sqrt{\alpha_k}\Big)
\prod_{k\in K}\ab{b_k}\nonumber\\
&=\frac{m!N^{N/2}}{(N-m)^{(N-m)/2}m^{m/2}}
\Coeff{y^m}{\prod_{k=1}^n(\ab{b_k}+\sqrt{\alpha_k}y)}.
\end{align*}
The proof is easily completed using Lemma \ref{la124376}.
\end{proof}

\begin{lemma}\label{la287765}
For $r\in\Zpl$, $t,x\in[0,1]$, we have 
$\sum_{m=0}^r(m+1)^tx^m\leq(\frac{1-x^{r+1}}{1-x})^{1+t}$.
\end{lemma}
\begin{proof}
This follows from 
$\sum_{m=0}^r(m+1)x^m
=\frac{1-(r+2)x^{r+1}+(r+1)x^{r+2}}{(1-x)^2}
\leq(\frac{1-x^{r+1}}{1-x})^{2}$ and 
H\"older's inequality, i.e.\ 
$\sum_{m=0}^r(m+1)^tx^m\leq
(\sum_{m=0}^r(m+1)x^m)^{t}(\sum_{m=0}^rx^m)^{1-t}$. 
\end{proof}

\noindent
The following lemma is more precise than Lemma 3 in \cite{roo00}.
\begin{lemma}\label{la4866}
Let $\ell,m,N\in\NN$, $\ell\leq m\leq N$ and
$C_\ell= \big(\frac{\ee^\ell\ell!}{\ell^{\ell+1/2}}\big)^{1/2}$. Then
\begin{align}\label{eq1987}
\frac{N^{N/2}}{(N-m)^{(N-m)/2}m^{m/2+1/4}\binomial{N}{m}^{1/2}}
\leq C_\ell.
\end{align}
\end{lemma}
\begin{proof}
Let $p(m,N)=\frac{N^{N}}{(N-m)^{N-m}\binomial{N}{m}}$.
Since $q(k):=(\frac{k}{k+1})^k$ is decreasing in $k\in\Zpl$, we have 
\begin{align*}
\frac{p(m,N)}{p(m,N+1)}=\frac{N^N(N+1-m)^{N+1-m}(N+1)}{(N+1)^{N+1}
(N-m)^{N-m}(N+1-m)}=\frac{q(N)}{q(N-m)}\leq 1.
\end{align*}
Hence 
$p(m,N)\leq \lim_{\widetilde{N}\to\infty}p(m,\widetilde{N})=\ee^m m!$ 
and therefore the left-hand side of (\ref{eq1987}) is bounded by
$(\frac{\ee^mm!}{m^{m+1/2}})^{1/2}$. Since this is 
decreasing in $m$ (cf.\ Mitrinovi\'{c} \cite[p.\ 183]{mit70}), 
the assertion follows. 
\end{proof}

\noindent
We now present our first main result, which generalizes Theorem A. 
Indeed, it will turn out that $\gamma\leq\frac{n}{N}$ and, for 
$\ell=1$, $H_\ell(Z)=\prod_{k=1}^n\widetilde{z}_k$, see the 
Remarks~\ref{rem474} and~\ref{rem76345} below. 
A further advantage of $\gamma$
is that it can be equal to zero, namely in the case 
$z_{j,k}=\widetilde{z}_k$ for all $j\in\set{N}$ and $k\in\set{n}$. 
We note that the singularity in (\ref{eq7435}) can be 
removed, see Theorem~\ref{th12} below. 

\begin{theorem}\label{th2994}
Let $N\in\NN$, $n\in\set{N}$, $\ell\in\set{n}$ and 
$Z=(z_{j,k})\in\CC^{N\times n}$. 
For $j\in\set{N}$ and $k\in\set{n}$, we assume that 
$\ab{z_{j,k}}\leq 1$ and set 
$\widetilde{z}_k=\frac{1}{N}\sum_{j=1}^Nz_{j,k}$,
$a_{j,k}=z_{j,k}-\widetilde{z}_k$, 
$U_j(x)=\sum_{k=1}^n a_{j,k}x_k$, 
where $x=(x_1,\dots,x_n)$ is an indeterminate. 
Further, let $C_\ell$ be as in Lemma~\ref{la4866},
\begin{align*}
\alpha&=\frac{1}{nN}\sum_{j=1}^N\sum_{k=1}^n\ab{a_{j,k}}^2
,\quad \beta=\frac{1}{n}\sum_{k=1}^n\ab{\widetilde{z}_k}^2,
\quad \gamma=\frac{n\alpha}{N}\min\Big\{n,\,\frac{1}{1-\beta}\Big\},\\
G_m(Z)&=\frac{(N-m)!}{(n-m)!N!}\Coeff{x_1\cdots x_n}{
\Big(\prod_{j=1}^N(1+U_j(x))\Big)
\Big(\sum_{k=1}^n\widetilde{z}_kx_k\Big)^{n-m}}
\end{align*}
for $m\in\setn{n}$ and set $H_\ell(Z)=\sum_{m=0}^\ell G_m(Z)$.
If $\gamma<1$, then
\begin{align}\label{eq7435}
\Ab{\frac{(N-n)!}{N!}\Per{Z} -H_\ell(Z)}
\leq (\ell+1)^{1/4}C_{\ell+1}\,
\frac{\gamma^{(\ell+1)/2}}{(1-\gamma)^{3/4}}.
\end{align}
\end{theorem}
\begin{proof}
Let $W_m(x)=\coeff{y^m}{\prod_{j=1}^N(1+U_j(x)y)}$ for 
$m\in\setn{N}$. In view of (\ref{eq1326}),
\begin{align*}
\lefteqn{\prod_{j=1}^N\Big(1+\sum_{k=1}^nz_{j,k}x_k\Big)
=\prod_{j=1}^N\Big(U_j(x)+1+\sum_{k=1}^n\widetilde{z}_kx_k\Big)}\\
&=\sum_{m=0}^NW_m(x)\Big(1+\sum_{k=1}^n\widetilde{z}_kx_k\Big)^{N-m}
=\sum_{m=0}^N\sum_{r=0}^{N-m}\binomial{N-m}{r}W_m(x)
\Big(\sum_{k=1}^n\widetilde{z}_kx_k\Big)^{r},
\end{align*}
and
\begin{align}\label{eq3376}
G_m(Z)&=\frac{(N-m)!}{(n-m)!N!}\Coeff{x_1\cdots x_n}{W_m(x)
\Big(\sum_{k=1}^n\widetilde{z}_kx_k\Big)^{n-m}},
\end{align}
we see that
\begin{align*}
\Per{Z}
&=\sum_{m=0}^n\binomial{N-m}{n-m}
\Coeff{x_1\cdots x_n}{W_m(x)
\Big(\sum_{k=1}^n\widetilde{z}_kx_k\Big)^{n-m}}\\
&=\frac{N!}{(N-n)!}\sum_{m=0}^nG_m(Z)
\end{align*}
and therefore $\frac{(N-n)!}{N!}\Per{Z}=H_n(Z)$. 
Using Lemmas~\ref{la31765} and \ref{la4866} and the simple inequality 
$\binomial{n}{m}\leq\binomial{N}{m}(\frac{n}{N})^m$ 
for $m\in\setn{n}$, 
we obtain
\begin{align*}
\lefteqn{
\Ab{\frac{(N-n)!}{N!}\Per{Z}-H_\ell(Z)}
\leq \sum_{m=\ell+1}^n\ab{G_m(Z)}}\\
&\leq\sum_{m=\ell+1}^n
\frac{N^{N/2}}{(N-m)^{(N-m)/2}m^{m/2}}
\frac{\binomial{n}{m}}{\binomial{N}{m}}\alpha^{m/2}\beta^{(n-m)/2}\\
&\leq C_{\ell+1}\sum_{m=\ell+1}^n
\frac{\binomial{n}{m}}{\binomial{N}{m}^{1/2}}m^{1/4}
\alpha^{m/2}\beta^{(n-m)/2}\\
&\leq C_{\ell+1} \sum_{m=\ell+1}^n m^{1/4}\gamma^{m/2}
\Big(\binomial{n}{m}\max\Big\{1-\beta, \frac{1}{n} \Big\}^m
\beta^{n-m}\Big)^{1/2},
\end{align*}
where we used that $\beta\in[0,1]$. By applying Cauchy's inequality
and the fact that, since $\ell\geq1$, 
$\sum_{m=\ell+1}^n\binomial{n}{m}\frac{1}{n^{m}}\leq 
(1+\frac{1}{n})^n-2\leq\ee-2< 1$, we obtain 
\begin{align*}
\Ab{\frac{(N-n)!}{N!}\Per{Z}-H_\ell(Z)}
&\leq C_{\ell+1} \Big(\sum_{m=\ell+1}^n
\sqrt{m}\gamma^{m}\Big)^{1/2}\\
&\leq (\ell+1)^{1/4}C_{\ell+1}\gamma^{(\ell+1)/2}
\Big(\sum_{m=0}^{n-\ell-1} \sqrt{m+1}\gamma^{m}\Big)^{1/2}.
\end{align*}
It remains to use Lemma \ref{la287765} with $t=\frac{1}{2}$.  
\end{proof}
\noindent
For the rest of the paper, 
let the notation of Theorem~\ref{th2994} hold. 

\begin{remark}\label{rem474} 
We have $\gamma\leq \frac{n}{N}$, since
\begin{align}\label{eq217}
\alpha=\frac{1}{nN}\sum_{j=1}^N\sum_{k=1}^n
\ab{z_{j,k}}^2-\beta\leq 1-\beta.
\end{align}
In particular, if $\ab{z_{j,k}}=1$ for all $j\in\set{N}$ and 
$k\in\set{n}$, then $\alpha=1-\beta$. Indeed, writing 
$z_{j,k}=u_{j,k}+\ii v_{j,k}$ and 
$\widetilde{z}_k=\widetilde{u}_k+\ii\widetilde{v}_k$ with 
$u_{j,k},v_{j,k}\in\RR$, 
$\widetilde{u}_k=\frac{1}{N}\sum_{j=1}^Nu_{j,k}$ and
$\widetilde{v}_k=\frac{1}{N}\sum_{j=1}^Nv_{j,k}$,  we obtain
\begin{align*}
\alpha
&=\frac{1}{nN}\sum_{j=1}^N\sum_{k=1}^n((u_{j,k}-\widetilde{u}_k)^2
+(v_{j,k}-\widetilde{v}_k)^2)\\
&=\frac{1}{nN}\sum_{k=1}^n\Big(\sum_{j=1}^N
(u_{j,k}^2+v_{j,k}^2)-N(\widetilde{u}_k^2+\widetilde{v}_k^2)\Big),
\end{align*}
from which (\ref{eq217}) follows. 
\end{remark}
Let us now collect some properties of the first few $G_m(Z)$, 
where we always assume that $m\in\setn{n}$. 

\begin{remark}\label{rem76345}
 
The first few $G_m(Z)$ can be evaluated as follows:
\begin{equation}\label{eq1756}
\begin{split}
G_0(Z)&=\prod_{k=1}^n\widetilde{z}_k,\quad
G_1(Z)=0,\\
G_2(Z)&=- \frac{(N-2)!}{N!}\sum_{K\subseteq\set{n}:\,\card{K}=2}
\Big(\sum_{j=1}^N\prod_{k\in K}a_{j,k}\Big)
\prod_{k\in\set{n}\setminus K}\widetilde{z}_k,
 \\
G_3(Z)&=2\frac{(N-3)!}{N!}\sum_{K\subseteq\set{n}:\,\card{K}=3}
\Big(\sum_{j=1}^N\prod_{k\in K}a_{j,k}\Big)
\prod_{k\in\set{n}\setminus K}\widetilde{z}_k. 
\end{split}
\end{equation}
In order to prove this, let 
\begin{align*}
V_m(x)=\sum_{j=1}^N(-U_j(x))^m,\quad
W_m(x)=\Coeff{y^m}{\prod_{j=1}^N(1+U_j(x)y)}
\end{align*}
for $m\in\setn{N}$. We have 
\begin{align*}
W_m(x)=-\frac{1}{m}\sum_{k=0}^{m-2}W_k(x)
V_{m-k}(x)\quad \mbox{ for }m\in\set{N},
\end{align*}
which can be shown in the same way as (10) in \cite{roo00}.
In particular, 
\begin{equation}\label{eq21756}
\begin{split}
W_0(x)&=1,\quad W_1(x)=0,\quad W_2(x)=-\frac{1}{2}V_2(x),
\\
W_3(x)&=-\frac{1}{3}V_3(x),\quad W_4(x)=\frac{1}{8}(V_2(x))^2
-\frac{1}{4}V_4(x). 
\end{split}
\end{equation}
In view of (\ref{eq3376}), (\ref{eq21756}) 
and
\begin{align}
\lefteqn{\Coeff{x_1\cdots x_n}{V_m(x)\Big(\sum_{k=1}^n
\widetilde{z}_kx_k\Big)^{n-m}}}\nonumber\\
&=(-1)^m\sum_{j=1}^N\Coeff{x_1\cdots x_n}{
(U_j(x))^m\Big(\sum_{k=1}^n\widetilde{z}_kx_k\Big)^{n-m}}\nonumber\\
&=(-1)^m(n-m)!\,m!\sum_{j=1}^N\sum_{K\subseteq\set{n}:\,\card{K}=m}
\Big(\prod_{k\in K}a_{j,k}\Big)
\prod_{k\in\set{n}\setminus K}\widetilde{z}_k,\label{eq17216}
\end{align}
for $m\in\set{n}$, we see that (\ref{eq1756}) is true. We note that 
the representations in (\ref{eq1756}) of $G_2(Z)$ and $G_3(Z)$ have 
a simple form, but the omitted ones of $G_m(Z)$ with $m\geq 4$ are 
more complicated.

From the above, we obtain that 
$H_1(Z)=\prod_{k=1}^n\widetilde{z}_k$ and, if $n\geq 2$,
\begin{equation}\label{eq71977}
H_2(Z)=\prod_{k=1}^n\widetilde{z}_k-\frac{1}{N(N-1)}
\sum_{K\subseteq\set{n}:\,\card{K}=2}
\Big(\sum_{j=1}^N\prod_{k\in K}a_{j,k}\Big)
\prod_{k\in\set{n}\setminus K}\widetilde{z}_k.
\end{equation}

\end{remark}

\begin{remark}
Let us derive some bounds for $\ab{G_2(Z)}$ and $\ab{G_3(Z)}$. 
From (\ref{eq17216}) and Lemma~\ref{la124376} 
it follows that, for $m\in\set{n}$, 
\begin{align*}
\lefteqn{\Ab{\Coeff{x_1\cdots x_n}{V_m(x)\Big(\sum_{k=1}^n
\widetilde{z}_kx_k\Big)^{n-m}}}}\\
&\leq(n-m)!\,m!\sum_{j=1}^N
\Ab{\Coeff{y^m}{\prod_{k=1}^n(\widetilde{z}_k+a_{j,k}y)}}\\
&\leq n!\sum_{j=1}^N\Big(\frac{1}{n}\sum_{k=1}^n
\ab{a_{j,k}}^2\Big)^{m/2}\beta^{(n-m)/2},
\end{align*}
which together with (\ref{eq1756}) gives
\begin{align*}
\ab{G_2(Z)}&\leq \frac{n(n-1)}{2(N-1)}\alpha\beta^{(n-2)/2},\\
\ab{G_3(Z)}&\leq \frac{1}{3}\frac{n!(N-3)!}{(n-3)!N!} 
\sum_{j=1}^N\Big(\frac{1}{n}\sum_{k=1}^n
\ab{a_{j,k}}^2\Big)^{3/2}\beta^{(n-3)/2}.
\end{align*}

The inequalities given above
can be used to derive bounds for 
$\ab{G_2(Z)}$ and $\ab{G_3(Z)}$ depending on $\gamma$. 
For precise calculations, we use the notation 
\begin{align*}
\gamma(d)=\frac{n\alpha}{N}\min\Big\{dn,\,\frac{1}{1-\beta}\Big\}
\end{align*} 
for $d\in(0,\infty)$, giving $\gamma=\gamma(1)$. We have 
\begin{align}
\ab{G_2(Z)}&\leq \gamma(1/2)
\max\Big\{\beta^{(n-2)/2},\,\frac{1}{2}n(1-\beta)\beta^{(n-2)/2}
\Big\}\frac{N(n-1)}{(N-1)n}\leq \gamma(1/2),\label{eq4236a}\\
\ab{G_3(Z)}&\leq \sqrt{3}\frac{(n-1)(n-2)N^2}{(N-1)(N-2)n^2}
\sum_{j=1}^N\Big(
\frac{1}{N^2}\sum_{k=1}^n\ab{a_{j,k}}^2\min\Big\{\frac{n}{3},
\frac{1}{1-\beta}\Big\}\Big)^{3/2}\nonumber\\
&\quad{}\times\max\Big\{\beta^{(n-3)/2},\,
\Big(\frac{n}{3}\Big)^{3/2}
(1-\beta)^{3/2}\beta^{(n-3)/2}\Big\}\nonumber\\
&\leq \sqrt{3}\sum_{j=1}^N\Big(
\frac{1}{N^2}\sum_{k=1}^n\ab{a_{j,k}}^2\min\Big\{\frac{n}{3},
\frac{1}{1-\beta}\Big\}\Big)^{3/2}.\label{eq4236}
\end{align}
We note that (\ref{eq4236}) implies that 
$\ab{G_3(Z)}$ is bounded by $\sqrt{3}(\gamma(1/3))^{3/2}$, 
which is however of worse order.

\end{remark}

\noindent
The following result is a consequence of Theorem~\ref{th2994}, 
(\ref{eq4236a}) and (\ref{eq4236}).
\begin{corollary}\label{cor387}
If $\gamma<1$, then
\begin{align}
\Ab{\frac{(N-n)!}{N!}\Per{Z} -\prod_{k=1}^n\widetilde{z}_k}
&\leq \gamma(1/2)
+\frac{3^{1/4}C_{3}\,\gamma^{3/2}}{(1-\gamma)^{3/4}},
\label{eq81776}\\
\Ab{\frac{(N-n)!}{N!}\Per{Z} -H_2(Z)}
&\leq \sqrt{3}\sum_{j=1}^N\Big(
\frac{1}{N^2}\sum_{k=1}^n\ab{a_{j,k}}^2\min\Big\{\frac{n}{3},
\frac{1}{1-\beta}\Big\}\Big)^{3/2}\nonumber\\
&\quad{}+\frac{2^{1/2}C_{4}\,\gamma^{2}}{(1-\gamma)^{3/4}},
\label{eq81777}
\end{align}
where the second inequality requires $n\geq 2$.
\end{corollary}
\begin{proof}[Proof of Theorem \ref{th363}]
Inequality (\ref{eq71665}) follows from (\ref{eq81776}) and
(\ref{eq892762}), while (\ref{eq71666}) can be easily be shown using
(\ref{eq81777}), (\ref{eq71977}) and the representation 
\begin{align*}
\lefteqn{\int f\,\dd(P^{Y_n}-Q_2)
=\int\Big(\frac{(N-n)!}{N!}\sum_{j\in\set{N}_{\neq}^n}
\prod_{k=1}^nf_k(X_{j_k}(\omega))-\prod_{k=1}^n\zeta_k(\omega)}
\qquad\\
&{}+\frac{1}{N(N-1)} 
\sum_{K\subseteq\set{n}:\,\card{K}=2}\sum_{j=1}^N
&\prod_{k\in K}\Big( f_k(X_j(\omega))-\zeta_k(\omega)\Big) 
\prod_{k\in\set{n}\setminus K}\zeta_k(\omega)\Big)\,\dd P(\omega) 
\end{align*}
for $f=\bigotimes_{k=1}^n f_k\in\calf_n$, where 
$\zeta_k(\omega)=\frac{1}{N}\sum_{j=1}^Nf_k(X_j(\omega))$.
\end{proof}

\noindent
We now show that the singularity in (\ref{eq7435}) can be removed. 
\begin{theorem}\label{th12}
For fixed $\ell\in\set{n}$, let $\kappa_\ell$ be the smallest 
absolute constant such that, without any restrictions on $\gamma$, 
\begin{align*}
\Ab{\frac{(N-n)!}{N!}\Per{Z} -H_\ell(Z)}
\leq \kappa_\ell\,\gamma^{(\ell+1)/2}.
\end{align*}
Then  
$\kappa_\ell\leq \frac{(\ell+1)^{1/4}C_{\ell+1}}{(1-x_\ell)^{3/4}}$,
where 
$x_\ell\in(0,1)$ is the unique positive solution of the 
equation 
\begin{align}\label{eq8546}
2+2^{1/4}C_{2}x\Big(\frac{1-x^{\ell-1}}{1-x}\Big)^{3/4}
=(\ell+1)^{1/4}C_{\ell+1}\,\frac{x^{(\ell+1)/2}}{(1-x)^{3/4}},
\quad(x\in(0,1)).
\end{align}
In particular, 
$\kappa_1\leq 3.57$, $\kappa_2\leq 5.53$ and $\kappa_3\leq7.08$.
\end{theorem}
\begin{proof}
Dividing (\ref{eq8546}) by $x^{(\ell+1)/2}$ yields
a decreasing left-hand side, whereas the right-hand side 
remains increasing in $x$. 
Therefore (\ref{eq8546}) has indeed a unique positive 
solution $x_\ell\in(0,1)$. 
Similarly as in the proof of Theorem~\ref{th2994}, we have 
\begin{align*}
\lefteqn{\Ab{\frac{(N-n)!}{N!}\Per{Z} -H_\ell(Z)}
\leq\frac{(N-n)!}{N!}\ab{\Per{Z}}+\ab{H_\ell(Z)}
\leq2+\sum_{m=2}^\ell\ab{G_m(Z)}}\\
&\leq2+2^{1/4}C_{2}\gamma
\Big(\sum_{m=0}^{\ell-2}\sqrt{m+1}\gamma^{m}\Big)^{1/2} 
\leq2+2^{1/4}C_{2}\gamma
\Big(\frac{1-\gamma^{\ell-1}}{1-\gamma}\Big)^{3/4}
=:h(\gamma).
\end{align*}
If $\gamma\in[0,x_\ell]$, we obtain 
by Theorem~\ref{th2994} that
\begin{align*}
\Ab{\frac{(N-n)!}{N!}\Per{Z} -H_\ell(Z)}
&\leq(\ell+1)^{1/4}C_{\ell+1}\,
\frac{\gamma^{(\ell+1)/2}}{(1-\gamma)^{3/4}}
\leq\frac{(\ell+1)^{1/4}C_{\ell+1}}{(1-x_\ell)^{3/4}}\,
\gamma^{(\ell+1)/2}.
\end{align*}
If $\gamma\in(x_\ell,\infty)$, then 
\begin{align*}
\Ab{\frac{(N-n)!}{N!}\Per{Z} -H_\ell(Z)}
&\leq h(\gamma)
\leq \frac{h(x_\ell)}{x_\ell^{(\ell+1)/2}}\gamma^{(\ell+1)/2} 
=\frac{(\ell+1)^{1/4}C_{\ell+1}}{(1-x_\ell)^{3/4}}\,
\gamma^{(\ell+1)/2}.
\end{align*}
It remains to use that 
$x_1\leq 0.5611$, $x_2\leq0.7222$ and $x_3\leq 0.7812$.  
\end{proof}
\section*{Acknowledgment}
The author is indebted to an anonymous reviewer for bringing to 
his attention the result given in Lemma \ref{la61966}
and for the indication of how it helps to improve the previous 
version of Lemma \ref{la2655}. This led to the significant improvement
of constants in several upper bounds. The author also thanks
Lutz Mattner for helpful comments. 


\bibliographystyle{amsalpha}

\end{document}